\documentclass[12pt,a4paper, reqno]{amsart}
\usepackage[utf8]{inputenc}

\usepackage{latexsym}

\calclayout
 
\usepackage{graphicx} 
\usepackage{amsmath}  
\usepackage{amssymb}  
\usepackage{amsthm} 
\usepackage{xcolor}

\usepackage{amsfonts} 
\usepackage{mathtools} 
\usepackage{physics}  
\usepackage{bm}       
\usepackage{cancel}   
\usepackage{bbm}      
\usepackage{nicefrac} 
\usepackage{siunitx}  
\usepackage{xfrac}    

\theoremstyle{plain}
\newtheorem{theorem}{Theorem}
\newtheorem{lemma}[theorem]{Lemma}

\newtheorem{definition}[theorem]{Definition}
\newtheorem{remark}[theorem]{Remark}

\numberwithin{theorem}{section}
\numberwithin{equation}{section}
\allowdisplaybreaks

\DeclareMathOperator*{\esslim}{ess\, lim}
\DeclareMathOperator{\sgn}{sgn}

\def\pa{\partial}

\begin{document}

\title[Multidimensional Scalar Conservation Laws]{MULTIDIMENSIONAL SCALAR CONSERVATION LAWS WITH NON-ALIGNED DISCONTINUOUS FLUX AND SINGULARITY OF SOLUTIONS
}

\author{AJLAN ZAJMOVI\' C }
\address{Faculty of Science and Mathematics, University of Montenegro,\\
Cetinjski put bb,
81000 Podgorica, Montenegro \\ 
}
\email{ajlan.z@ucg.ac.me} 

\begin{abstract}
We prove that the family of solutions to vanishing viscosity approximation for multidimensional scalar conservation laws with discontinuous non-aligned flux and zero initial data in the limit generates a singular measure supported along the discontinuity surface. 
\end{abstract}


\maketitle

\section{Introduction}

In the current contribution, we analyse a multi-dimensional extension of the scalar conservation laws with spatially discontinuous flux and zero initial data and aim to prove that its vanishing viscosity approximation generates a family of approximate solutions whose limit admits a singular measure supported on the discontinuity interface. By means of a simple change of unknown function (see \eqref{riemann} below), we can equally well consider a Riemann-type problem. Specifically, we consider
\begin{equation}
\label{eqmain}
\begin{split}
\partial_t u &+ \sum\limits_{k=1}^{d}\partial_{x_k}\left(f^k_L(\hat{x}_k,u) H(-x_1+\varphi(\hat{x}_1)) + f^k_R(\hat{x}_k,u) H(x_1-\varphi(\hat{x}_1))\right) = 0, 
\end{split}
\end{equation}  subject to the zero initial data $u|_{t=0} = u_0(x)=0$. Here, $u: \mathbb{R}_+ \times \mathbb{R}^d \to \mathbb{R}$ is the scalar unknown, and 
\begin{itemize}
    \item $\hat{x}_k=(x_1,\dots, x_{k-1},x_{k+1},\dots,x_d)\in \mathbb{R}^{d-1}$, $t\geq 0$;
    \item $H(z)=\begin{cases}
        1, & z\geq0\\
        0, & z < 0
    \end{cases}$ is the Heaviside function;
    \item the discontinuity manifold is given as the equipotential surface
$$
D=\{ x\in \mathbb{R}^d: \, x_1=\varphi(\hat{x}_1), \ \ \varphi\in C^2(\mathbb{R}^{d-1})\};
$$

\item  the vector fields $f_L=(f_L^1,\dots, f_L^d):\mathbb{R}^{d}\to\mathbb{R}^d$ and $f_R=(f_R^1,\dots, f_R^d) : \mathbb{R}^{d}\to\mathbb{R}^d$, for $k=1,\dots, d$, are bounded, continuously differentiable functions satisfying the {\bf non-alignment} condition:
\begin{equation}
   \label{(uslov)}
   \begin{split}
&\int\limits_{\mathbb{R}^{d-1}}\big(\max_{\lambda \in \mathbb{R}} f^k_L(\hat{x}_k,\lambda) - \min_{\lambda \in \mathbb{R}} f^k_R(\hat{x}_k, \lambda) \big)\, d\hat{x}_k\leq c<0, \ \ {\rm and} \\
&  \ \ f^k_L(\hat{x}_k,0) \leq f^k_R(\hat{x}_k,0) \ \  \forall k=1,\dots,d;
\end{split}
\end{equation} 

\item we assume that the flux functions \( f_{L}^k \) and \( f_{R}^k \), $k=1,\dots,d$, belong to \( L^{\infty}(\mathbb{R}^{d}) \), and their restrictions with respect to the transverse variables satisfy 
\begin{equation} 
\label{L1cond}
\sup\limits_{z\in \mathbb{R}}|f_{L,R}^k(\cdot_{\hat{x}_k},z)| \in L^1(\mathbb{R}^{d-1}). 
\end{equation}

\end{itemize}
 
Moreover, we assume that $\varphi$ is monotonic with respect to all the variables (see Figure \ref{Fig1}). For simplicity, we assume that for every $k=2,\dots,d$
\begin{equation}
\label{decreasing}
\frac{\partial \varphi}{\partial x_k}<0.
\end{equation} If, on the other hand, there exists some \( l \in \{2, \dots, d\} \) such that \( \frac{\partial \varphi (\hat{x}_1) }{\partial x_l} > 0 \), instead of \eqref{(uslov)}, we would require the opposite inequality with $min$ and $max$ interchanged:
\begin{equation}
   \label{(uslov)-1}
\int\limits_{\mathbb{R}^{d-1}}\big(\min_{\lambda \in \mathbb{R}} f^l_L(\hat{x}_l,\lambda) - \max_{\lambda \in \mathbb{R}} f^l_R(\hat{x}_l, \lambda) \big) d\hat{x}_k\geq c>0.
\end{equation}

 We note that in one-dimensional case, condition \eqref{(uslov)} reduces to (see \cite[(1.2)]{karlsen2025})
$$
\max_{\lambda \in \mathbb{R}} f_L(\lambda) - \min_{\lambda \in \mathbb{R}} f_R(\lambda)<0.
$$

\begin{figure}[h]
\label{Fig1}    \centering
    \includegraphics[width=0.6\textwidth]{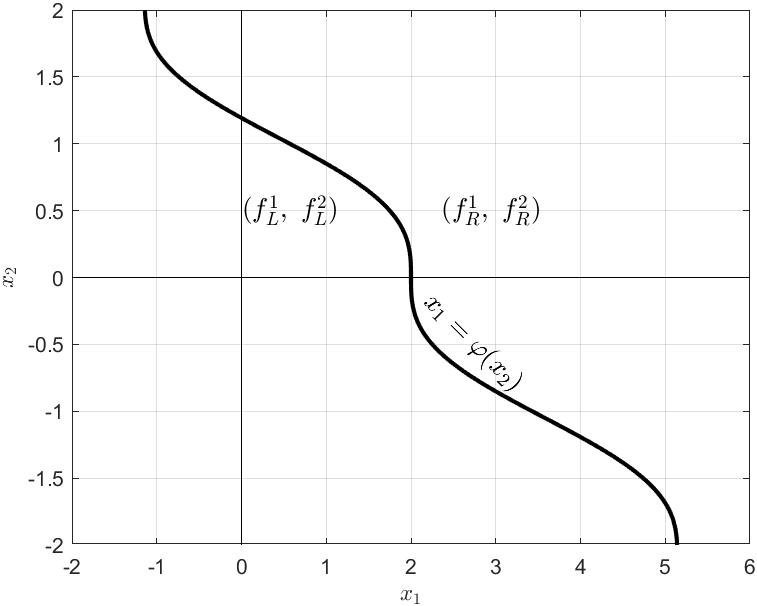}
    \caption{Schematic representation of the flux discontinuity surface in 2D case}
    \label{fig:jmf}
\end{figure}

We conclude the first part of the introduction by noting that, instead of the zero initial data, we can equally well consider the Riemann initial condition
\begin{equation}
\label{riemann}
u|_{t=0} = u_0(x) :=
\begin{cases}
u_L, & x_1 < \varphi(\hat{x}_1), \\
u_R, & x_1 > \varphi(\hat{x}_1),
\end{cases}
\end{equation}
where $u_L, u_R \in \mathbb{R}$ are constant states.

Indeed,  by introducing the substitution \( v = u - u_0 \), the new initial condition becomes \( v_0 = 0 \) and we are back to the original situation (with a different flux still satisfying \eqref{(uslov)}). 

Note that the non-alignment condition~\eqref{(uslov)} rules out the trivial solution $u_\varepsilon \equiv 0$, since the systematic mismatch between the left and right fluxes necessarily generates a singular concentration on the discontinuity interface.

\subsection{Motivation and Context}

Conservation laws with discontinuous fluxes naturally arise in various applications, including multi-phase flows in porous media, traffic flow with bottlenecks or varying road conditions, sedimentation and clarification processes, transport phenomena in heterogeneous materials, as well as in geophysical and engineering models \cite{burger2003,burger2009, bustos1999,lighthill1955}. These problems are also often studied in the context of convergence and stability of numerical methods (see e.g.  \cite{karlsen2023primjena} and references therein). The formulation in \eqref{eqmain} generalizes the classical Riemann problem to settings with sharp interfaces, where the flux law switches discontinuously across a hyperplane. As noted by Abreu et al. \cite{Abreu2024}, such settings may lead to shocks that do not satisfy classical Lax’s conditions and may require weak or measure-valued solutions.

Recently, considerable attention has been devoted to conservation laws with discontinuous flux functions, especially those involving so-called ``distanced fluxes.'' A notable model of this type was introduced by Adimurthi, Mishra, and Gowda \cite{adimurthi2007}, who studied conservation laws with spatially variable discontinuities in the flux. In their framework, the point \( x = 0 \) acts as an attractor toward which the fluid or gas flows. This behavior is captured by rewriting the conservation law in a non-conservative form, leading to a nonlinear transport equation with a singular source term of the form \( (f(u) - g(u))\delta(x) \), where \( u \) denotes the density of the fluid or gas. The consistently non-negative term \( (f(u) - g(u))\delta(x) \) effectively increases the density at \( x = 0 \), modeling a localized gain of mass or energy.

Although no direct physical applications of such equations are currently known, they may offer a simplified framework for modeling certain astrophysical phenomena. In particular, such formulations may be relevant in cosmological settings where space itself is treated as a compressible fluid. One prominent example is the (generalized) Chaplygin gas model, which attempts to unify dark energy and dark matter as manifestations of a single exotic fluid characterized by an unconventional equation of state. Discontinuities in the flux can be interpreted as the result of interactions between large-scale cosmic structures, such as galaxies and voids. Moreover, Dirac delta solutions corresponds to localized high-density events in the universe, potentially representing phenomena such as black hole formation.

In this context, the model considered here can be viewed as a simplified version of the generalized Chaplygin gas equations. A rigorous mathematical treatment of such models, including results on uniqueness and the structure of solutions involving singularities, is given by \cite{nedeljkov2017} (see also references therein).

In one spatial dimension, the theory of such problems is well developed, yet even in that setting, the presence of a discontinuous flux may lead to non-classical solutions, such as $\delta$-shocks or other singular measures. In fact, under the above inequality, there may exist no weak solution satisfying standard entropy conditions, necessitating the consideration of measure-valued solutions supported at the interface. This phenomenon has been observed in earlier works \cite{diehl1995, gimse1991, isaacson1992}, where admissibility criteria such as minimal jump and variation conditions were introduced. Similar problems have been studied in \cite{adimurthi2003, adimurthi2005, andreianov2011}, where non-uniqueness of entropy solutions and failure of Rankine–Hugoniot consistency were rigorously addressed.

The class of flux discontinuities considered here generalizes one-dimensional models such as those studied in \cite{karlsen2025, krunic2021}, where a spatially discontinuous flux leads to the formation of singular solutions that cannot be interpreted as classical weak solutions. In the one-dimensional setting, the interplay between discontinuous fluxes and interface geometry has been shown to produce nontrivial phenomena, including measure-valued solutions involving Dirac deltas.

The occurrence of such singularities has been rigorously examined first in \cite{krunic2021} via the shadow wave approach (see also \cite{nedeljkov2010}) and then in the case of the vanishing viscosity approach \cite{karlsen2025} as we are doing here in the multidimensional situation. We note that our contribution essentially relies on the adaptation of techniques developed in \cite{karlsen2025}.  

In both papers \cite{karlsen2025, krunic2021}, the limiting solutions with Dirac delta components supported at the interface were discovered. These techniques rely on regularized flux approximations and smooth Heaviside transitions. The resulting limits, although singular, satisfy a generalized weak formulation with an added measure term that compensates for the flux imbalance across the interface. We address readers to \cite{dalmaso, danilov2005, kalisch2012, kalisch2022, nedeljkov2004} for possible approaches on formalizing the weak solutions concept. 

We note also that the weak asymptotic techniques were used in describing singular limits in conservation laws. These include the works \cite{danilov2003, danilov2008, shelkovich2005}, where interactions of nonlinear waves and $\delta$-shock dynamics are studied. Moreover, Colombeau-type generalized function frameworks \cite{colombeau1984, nedeljkov2004} have been successfully applied in capturing multiplication of distributions, allowing rigorous definitions of nonlinear operations involving Dirac measures.

\subsection{Overview and Main Contributions}

Previous studies have shown that, under the non-alignment condition \eqref{(uslov)}, classical entropy solutions to the one-dimensional problem may not exist. Instead, weak solutions may contain singular measures supported at the interface. In particular, $\delta$-shock solutions have been derived using shadow wave approximations \cite{krunic2021} and vanishing viscosity \cite{karlsen2025}. We note however that such a problem has not been considered in the multidimensional setting.

To this end, we extend the vanishing viscosity approach to the multi-dimensional problem \eqref{eqmain}, considering the regularized equation
\begin{align}
\label{viscosity}
\partial_t u_\varepsilon &+ \sum\limits_{k=1}^{d}\partial_{x_k}\Big(f^k_L(\hat{x}_k,u_\varepsilon) H_\varepsilon(-x_1+\varphi(\hat{x}_1)) \\& \qquad\qquad \ \ + f^k_R(\hat{x}_k,u_\varepsilon) H_\varepsilon(x_1-\varphi(\hat{x}_1))\Big) = \varepsilon\Delta u_\varepsilon, \notag \\
\label{icviscosity}
u_\varepsilon|_{t=0}&=0,
\end{align}
where $H_\varepsilon(x) = \omega(x/\varepsilon)$ is a smooth approximation of the Heaviside function and $\omega \in C^\infty(\mathbb{R})$ is a monotone transition function such that $H(x)-\omega(x/\varepsilon)$ is compactly supported. We demonstrate that the family $(u_\varepsilon)$ converges, in the sense of distributions, to a singular measure. More precisely, the following theorem is the main result of the paper.
\begin{theorem}
    \label{main-thm} The family of solutions $(u_\varepsilon)$ to \eqref{viscosity}, \eqref{icviscosity} is bounded in $L^1(\mathbb{R}^+\times \mathbb{R}^d)$. The weak subsequential  limit $u$ of $(u_\varepsilon)$ in the space of Radon measure ${\mathcal M}_{loc}(\mathbb{R}^+\times \mathbb{R}^d))$ is singular on the discontinuity interface $D$ i.e. measure of the set $D$ with respect to the measure $u$ is strictly less than zero
    $$
    u(D) < 0.
    $$
\end{theorem}

The paper is organized as follows. After the Introduction, we introduce notions and notations that we are going to use. In Section 3, we prove the main theorem of the paper. The paper is completed by concluding remarks where we try to explain a surprising fact that the singular measure naturaly arises although the described process is controlled by nonlinear functions not allowing a natural action on singular measures (such as e.g. $\delta^2$ for the Dirac distribution; compare with sticky particles system \cite{natile2010}, presureless gas dynamics \cite{bouchut1994}, chromatography system \cite{wang}, Chaplygin gas \cite{nedeljkov2017} etc.)

\section{Notions and Notation}

In this section, we collect the main symbols, functions, and analytical tools used throughout the paper. These include regularization techniques, distributional concepts, and functional analytic settings central to our treatment of scalar conservation laws with discontinuous fluxes in multi-dimensional space. The tools and conventions introduced here support the vanishing viscosity analysis, weak convergence framework, and singular limit passage developed in the subsequent sections.

\subsection{Heaviside Function and Mollifications}

The Heaviside function \( H:\mathbb{R} \to \{0,1\} \) is defined by
\[
H(x) = 
\begin{cases}
1, & x \geq 0, \\
0, & x < 0.
\end{cases}
\]
To approximate the discontinuity at the origin, we use a smooth mollified family \( \{H_\varepsilon\}_{\varepsilon > 0} \) such that
\begin{equation}
\label{izlgalacnihesijan}
H_\varepsilon(x) = \omega\left( \frac{x}{\varepsilon} \right),
\end{equation}
where \( \omega \in C^\infty(\mathbb{R}) \) is a monotone function satisfying \( \omega(z) = 0 \) for \( z \leq -1 \) and \( \omega(z) = 1 \) for \( z \geq 1 \). Then \( H_\varepsilon \to H \) pointwise as \( \varepsilon \to 0 \).

The distributional derivative of \( H_\varepsilon \), denoted by
\[
\delta_\varepsilon(x) := H_\varepsilon'(x),
\]
approximates the Dirac delta distribution \( \delta \) and we shall assume that $\delta_\varepsilon$ is an {\bf even function} for every $\varepsilon$ which is an essential assumption in the course of the proof. 

We recall the identity 
\[
\langle z \delta(z), \theta(z) \rangle = 0 \quad \forall \theta \in \mathcal{D}(\mathbb{R}).
\]

\subsection{Sign Functions and Their Variants}

The standard signum function is defined as
\begin{equation}
\sgn(x) = 
\begin{cases}
-1, & x < 0, \\
0, & x = 0, \\
1, & x > 0.
\end{cases}
\label{eq:sgn_function}
\end{equation}

We also introduce the one-sided variants:
\[
\sgn_{+}(x) = \begin{cases}
1, & x > 0, \\
0, & x \leq 0,
\end{cases}
\qquad
\sgn_{-}(x) = \begin{cases}
-1, & x < 0, \\
0, & x \geq 0.
\end{cases}
\]
These functions are used to express entropy conditions and local directional behavior near discontinuities.




\subsection{Flux Functions with Discontinuity Interfaces}

We consider fluxes that are discontinuous across a moving interface. For each spatial direction \( k = 1, \dots, d \), the flux \( f^k: \mathbb{R}^d \to \mathbb{R} \) is given by
\[
f^k(\hat{x}_k, u) =
\begin{cases}
f_L^k(\hat{x}_k, u), & x_1 < \varphi(\hat{x}_1), \\
f_R^k(\hat{x}_k, u), & x_1 > \varphi(\hat{x}_1),
\end{cases}
\]
where \( \varphi: \mathbb{R}^{d-1} \to \mathbb{R} \) defines a $C^2(\mathbb{R}^{d-1})$--interface and \( f_L^k, f_R^k \in C^1(\mathbb{R}^{d}) \) are spatially dependent flux components.

\subsection{Function Spaces}

The following function spaces are used throughout the analysis:
\begin{itemize}
    \item \( W^{2,2}(\mathbb{R}^d) \): the Sobolev space of functions with square-integrable second derivatives.
    \item \( L^2([0,T]; W^{2,2}(\mathbb{R}^d)) \): the Bochner space of time-dependent functions with values in \( W^{2,2}(\mathbb{R}^d) \), square-integrable in time.
    \item \( L^\infty([0,T]; L^1(\mathbb{R}^d)) \): functions essentially bounded in time, with values in \( L^1(\mathbb{R}^d) \).
    \item \( \mathcal{M}_{loc}((0,T) \times \mathbb{R}^d) \): the space of locally finite Radon measures, used to describe weak limits of sequences with possible concentration effects.
\end{itemize}

\subsection{Distributional Derivatives and Non-Smooth Terms}

Operations on quantities such as \( |u_\varepsilon| \), \( \sgn(u_\varepsilon) \), and \( \delta(x) \) are interpreted in the sense of distributions. All differentiation operations involving non-smooth functions are rigorously defined using test functions in \( \mathcal{D}(\mathbb{R}^d) \). Test functions are always assumed to be smooth with compact support in both time and space unless specified otherwise.

\subsection{Regularized Absolute Value}

To avoid non-differentiability at zero, we introduce the regularized absolute value \( |x|_\varepsilon \), defined as
\[
|x|_\varepsilon =
\begin{cases}
|x|, & |x| \ge \varepsilon, \\[1.2ex]
\displaystyle \frac{3}{8} \cdot \frac{x^6}{\varepsilon^5}
- \frac{5}{4} \cdot \frac{x^4}{\varepsilon^3}
+ \frac{15}{8} \cdot \frac{x^2}{\varepsilon}, & |x| < \varepsilon,
\end{cases}
\]
which is $C^2(\mathbb{R})$ and approximates \( |x| \) as \( \varepsilon \to 0 \). This function is used in estimates involving weak lower semicontinuity. { We will denote the derivative of} \( |x|_\varepsilon \) by \( \sgn_\varepsilon(x) \), i.e.,
\[
\sgn_\varepsilon(x) :=(|x|_\varepsilon)',
\]
which represents a \( C^2(\mathbb{R}) \) smooth approximation of the sign function defined in~\eqref{eq:sgn_function}.

\subsection{Sign $\lesssim$}
We write \( a \lesssim b \) if there exists a constant \( c > 0 \) independent of $a$ and $b$ such that \( a \leq b\cdot c \).

\subsection{Vector notation with $[\cdot]$}

By $[a,\hat{x}_k]$ we denote the vector of the form:
\begin{equation}
\label{k-notation}
[a,\hat{x}_k] = (x_1, \dots, x_{k-1}, a, x_{k+1}, \dots, x_d),
\end{equation}
that is, the vector obtained by replacing the $k$-th component of $x$ with $a$, while all other components remain unchanged.

\subsection{Weak Convergence Notation}

We write
\[
u_\varepsilon \overset{\star}{\rightharpoonup} u \quad \text{in } \mathcal{M}((0,T) \times \mathbb{R}^d)
\]
to denote weak-$\star$ convergence of \( u_\varepsilon \) to a measure-valued limit \( u \).

\subsection{Generic Constants and weak lower semicontinuity}

Throughout the paper, the symbol \( \tilde{C} \) denotes a generic positive constant that may vary from line to line but is always independent of \( \varepsilon \), unless stated otherwise.

\begin{definition}[Weak Lower Semicontinuity]
Let \( X \) be a Banach space, and let \( f : X \to \mathbb{R} \cup \{+\infty\} \) be a functional. Then \( f \) is said to be \emph{weakly lower semicontinuous} if, for every sequence \( \{x_n\} \subset X \) such that \( x_n \rightharpoonup x \) weakly in \( X \), it holds that
\[
f(x) \leq \liminf_{n \to \infty} f(x_n).
\]
\end{definition}

\begin{theorem}\label{Ajlan}[Weak Lower Semicontinuity for Measures]
Let $\{\mu_n\}_{n\in \mathbb{N}}$ be a sequence of Radon measures on a locally compact Hausdorff space $X$ that converges weakly* to a Radon measure $\mu$, i.e.,
\[
\mu_n \xrightarrow{w^*} \mu \quad \text{in } \mathcal{M}(X).
\]
Then for every lower semicontinuous function $f : X \to [0, +\infty]$, the following inequality holds:
\[
\int\limits_X f \, d\mu \leq \liminf_{n \to \infty} \int\limits_X f \, d\mu_n.
\]
\end{theorem}

\section{A vanishing viscosity approximation}

In this section, we consider the vanishing viscosity approximation \eqref{viscosity}, \eqref{icviscosity} of the original problem \eqref{eqmain} with zero initial data. We start with the following lemma.


\begin{lemma}\label{lema1}
For each fixed $\varepsilon > 0$, the parabolic problem \eqref{viscosity} with smooth flux and homogeneous initial data, has a unique solution satisfying
\begin{equation}
\label{regularity-eps}    
u_\varepsilon \in L^2\left([0, T]; W^{2,2}(\mathbb{R}^d)\right), \quad \text{for any } T > 0.
\end{equation}
The solution $u_\varepsilon$ is uniformly bounded in $L^\infty([0, T]; L^1(\mathbb{R}^d))$:
\begin{equation}
\label{L1-bnd}
    \|u_\varepsilon(t, \cdot)\|_{L^1(\mathbb{R}^d)} \leq \tilde{C}_T, \quad t \in [0, T],
\end{equation} where $\tilde{C}_T$ is a constant independent of $\varepsilon$, for any $T > 0$.\\
\end{lemma}

\textbf{Proof}
The existence of solution satisfying \eqref{regularity-eps} follows from \cite[Chapter 5]{Ladyzenskaja:1967fy}. 
To prove \eqref{L1-bnd}, we first note that for $k=1$
\begin{align}
\label{l1-1}
    \partial_{x_{1}}\Big(&\sgn(u_\varepsilon)\big((f^1_L(\hat{x}_1,u_\varepsilon)-f_L^1(\hat{x}_1,0)) H_\varepsilon(-x_1+\varphi(\hat{x}_1))  \\& \qquad\qquad + (f^1_R(\hat{x}_1,u_\varepsilon)-f_R^1(\hat{x}_1,0)) H_\varepsilon(x_1-\varphi(\hat{x}_1))\big)\Big)  \notag\\
    &= \sgn(u_\varepsilon)\partial_{x_1}\left(f^1_L(\hat{x}_k,u_\varepsilon) H_\varepsilon(-x_1+\varphi(\hat{x}_1)) + f^1_R(\hat{x}_1,u_\varepsilon) H_\varepsilon(x_1-\varphi(\hat{x}_1))\right) \notag\\
    &\qquad \qquad - \sgn(u_\varepsilon)(f_R^1(\hat{x}_1,0)-f_L^1(\hat{x}_1,0))\delta_\varepsilon(x_1-\varphi(\hat{x}_1)),
    \notag
\end{align} while for $k\neq 1$
\begin{align}
\label{l1-2}
    \partial_{x_{k}}\Big(&\sgn(u_\varepsilon)\big((f^k_L(\hat{x}_k,u_\varepsilon)-f_L^k(\hat{x}_k,0)) H_\varepsilon(-x_1+\varphi(\hat{x}_1))  \\&\qquad\qquad + (f^k_R(\hat{x}_k,u_\varepsilon)-f_R^k(\hat{x}_k,0)) H_\varepsilon(x_1-\varphi(\hat{x}_1))\big)\Big)  \notag\\
    &= \sgn(u_\varepsilon)\partial_{x_k}\left(f^k_L(\hat{x}_k,u_\varepsilon) H_\varepsilon(-x_1+\varphi(\hat{x}_1)) + f^k_R(\hat{x}_k,u_\varepsilon) H_\varepsilon(x_1-\varphi(\hat{x}_1))\right) \notag\\
    &\qquad \qquad - \sgn(u_\varepsilon)(f_L^k(\hat{x}_k,0)-f_R^L(\hat{x}_k,0))\delta_\varepsilon(x_1-\varphi(\hat{x}_1))\frac{\partial \varphi(\hat{x}_1)}{\partial x_k}.
    \notag
\end{align} Next, we multiply \eqref{viscosity} by $\sgn(u_\varepsilon)$. We have
\begin{align}
\label{(5)}
\partial_t |u_\varepsilon| &+ \sum\limits_{k=1}^{d}\sgn(u_\varepsilon)\partial_{x_k}\Big(f^k_L(\hat{x}_k,u_\varepsilon) H_\varepsilon(-x_1+\varphi(\hat{x}_1)) \\&\qquad\qquad+ f^k_R(\hat{x}_k,u_\varepsilon) H(x_1-\varphi(\hat{x}_1))\Big) = \varepsilon\Delta |u_\varepsilon|-2\varepsilon\delta(u_\varepsilon)|\nabla u_\varepsilon|^2.
\notag
\end{align}
Combining equations \eqref{l1-1}, \eqref{l1-2}, and \eqref{(5)} and taking into account that, in sense of distribution,  
$$
2\varepsilon\delta(u_\varepsilon)|\nabla u_\varepsilon|^2\geq 0,
$$   we get:
\begin{align}
\label{comb-345}
&\partial_t |u_\varepsilon| 
+ \partial_{x_{1}}\Big(\sgn(u_\varepsilon)\big((f^1_L(\hat{x}_1,u_\varepsilon)-f_L^1(\hat{x}_1,0)) H_\varepsilon(-x_1+\varphi(\hat{x}_1))  \\
&\qquad \qquad \ \  + (f^1_R(\hat{x}_1,u_\varepsilon)-f_R^1(\hat{x}_1,0)) H_\varepsilon(x_1-\varphi(\hat{x}_1))\big)\Big) \nonumber\\&
+ \sgn(u_\varepsilon)(f_R^1(\hat{x}_1,0)-f_L^1(\hat{x}_1,0))\delta_\varepsilon(x_1-\varphi(\hat{x}_1)) \notag \\&
+\sum\limits_{k=2}^{d}\partial_{x_{k}}\Big(\sgn(u_\varepsilon)\big((f^k_L(\hat{x}_k,u_\varepsilon)-f_L^k(\hat{x}_k,0)) H_\varepsilon(-x_1+\varphi(\hat{x}_1)) \notag \\
&\qquad\qquad \ \  + (f^k_R(\hat{x}_k,u_\varepsilon)-f_R^k(\hat{x}_k,0)) H_\varepsilon(x_1-\varphi(\hat{x}_1))\big)\Big) \nonumber \\
&+ \sum\limits_{k=2}^d \sgn(u_\varepsilon)(f_L^k(\hat{x}_k,0)-f_R^k(\hat{x}_k,0))\delta_\varepsilon(x_1-\varphi(\hat{x}_1))\frac{\partial \varphi(\hat{x}_1)}{\partial x_k} 
\leq \varepsilon\Delta |u_\varepsilon|.
\notag
\end{align} 
We denote $Q^d(K)=[-K,K]^d$. Integrating \eqref{comb-345} over $[0,T]\times Q^d(K)$ leads to: 

\begin{equation}\label{1.7}
\begin{aligned}
&\int\limits_{0}^{T}\int\limits_{Q^d(K)} \partial_t |u_\varepsilon|\,dx\,dt  \\
&\quad+ \int\limits_{0}^{T}\int\limits_{Q^d(K)}\partial_{x_{1}}\Big(\sgn(u_\varepsilon)\big((f^1_L(\hat{x}_1,u_\varepsilon)-f_L^1(\hat{x}_1,0)) H_\varepsilon(-x_1+\varphi(\hat{x}_1)) \\
&\qquad\qquad\qquad\qquad + (f^1_R(\hat{x}_1,u_\varepsilon)-f_R^1(\hat{x}_1,0)) H_\varepsilon(x_1-\varphi(\hat{x}_1))\big)\Big)\,dx\,dt \\
&\quad+ \int\limits_{0}^{T}\int\limits_{Q^d(K)}\sgn(u_\varepsilon)(f_R^1(\hat{x}_1,0)-f_L^1(\hat{x}_1,0))\delta_\varepsilon(x_1-\varphi(\hat{x}_1))\,dx\,dt \\
&\quad+ \sum\limits_{k=2}^{d}\int\limits_{0}^{T}\int\limits_{Q^d(K)}\partial_{x_{k}}\Big(\sgn(u_\varepsilon)\big((f^k_L(\hat{x}_k,u_\varepsilon)-f_L^k(\hat{x}_k,0)) H_\varepsilon(-x_1+\varphi(\hat{x}_1))  \\
&\qquad\qquad\qquad\qquad\qquad\qquad + (f^k_R(\hat{x}_k,u_\varepsilon)-f_R^k(\hat{x}_k,0)) H_\varepsilon(x_1-\varphi(\hat{x}_1))\big)\Big)\,dx\,dt \\
&\quad+ \sum\limits_{k=2}^d\int\limits_{0}^{T}\int\limits_{Q^d(K)} \sgn(u_\varepsilon)(f_L^k(\hat{x}_k,0)-f_R^k(\hat{x}_k,0))\delta_\varepsilon(x_1-\varphi(\hat{x}_1))\frac{\partial \varphi(\hat{x}_1)}{\partial x_k} \,dx\,dt \\
&\quad\leq \int\limits_{0}^{T}\int\limits_{Q^d(K)}\varepsilon\Delta |u_\varepsilon|\,dx\,dt.
\end{aligned}
\end{equation}

Since $u_\varepsilon \in L^2([0, T ]; W^{2,2}(\mathbb{R}^d))$, we have for all $k,j =1,\dots,d$
\begin{equation*}
    \begin{split}
        &\esslim\limits_{K\to \infty} u_\varepsilon(t, [\pm K, \hat{x}_k]):= \esslim\limits_{K\to \infty} u_\varepsilon(t,x_1,\dots,x_{k-1},\pm K,x_{k+1},\dots,x_d)=0,\\
        &\esslim\limits_{K\to \infty} \pa_{x_j}u_\varepsilon(t, [\pm K, \hat{x}_k]):= \esslim\limits_{K\to \infty} \pa_{x_j} u_\varepsilon(t,x_1,\dots,x_{k-1},\pm K,x_{k+1},\dots,x_d)=0,
    \end{split}
\end{equation*} and thus, by the continuity of the functions $f$ and $u_\varepsilon$, we have that (pointwise)
\begin{equation*}
\begin{split}
&f^k_L(\hat{x}_k,u_\varepsilon(t, [\pm K, \hat{x}_k]))-f_L^k(\hat{x}_k,0) \to 0, \quad \text{as } K \to +\infty;\\ 
&f^k_R(\hat{x}_k,u_\varepsilon(t, [\pm K, \hat{x}_k]))-f_L^R(\hat{x}_k,0) \to 0 \quad \text{as } K \to +\infty.
\end{split}
\end{equation*}
Also, without loss of generality we can assume that
\[
\int\limits_{\mathbb{R}} \delta_\varepsilon(x) \, dx \leq 1.
\]

Thus, using the Newton-Leibnitz formula and Lebesgue dominated convergence theorem, and letting $K \to +\infty$ in \eqref{1.7}, we get: 
\begin{align}
&\int\limits_{R^d} |u_\varepsilon|\,dx
\leq\sum\limits_{k=2}^{d}\int\limits_{0}^{T}\int\limits_{R^{d-1} }\sgn(u_\varepsilon)(f_L^k(\hat{x}_k,0)-f_R^k(\hat{x}_k,0))\times\\& \qquad\qquad\qquad\qquad\times\int\limits_{R}\delta_\varepsilon(x_1-\varphi(\hat{x}_1))\frac{\partial \varphi(\hat{x}_1)}{\partial x_k} \, dx_k\,d\hat{x}_k\,dt\nonumber \\
 & \ \ + \int\limits_{0}^{T}\int\limits_{R^{d-1}}\sgn(u_\varepsilon)(f_R^1(\hat{x}_1,0)-f_L^1(\hat{x}_1,0))\int\limits_R\delta_\varepsilon(x_1-\varphi(\hat{x}_1))\, dx_1\, d\hat{x}_1\,dt 
=0. \nonumber
\end{align} Introducing the change of variables in each of the integrals for $k=2,\dots,d$:
\begin{equation}
    \label{cov}
y_k=\varphi(\hat{x}_1), \ \ \hat{y}_k=\hat{x}_k,
\end{equation}  and taking into account
$$
\frac{\partial \varphi(\hat{x}_1)}{\partial x_k}<0,
$$we reach to
\begin{equation*}
    \begin{split}
        &\int\limits_{R^d} |u_\varepsilon|\,dx
\\&\leq\sum\limits_{k=2}^{d}\int\limits_{0}^{T}\int\limits_{R^{d-1} }\sgn(\tilde{u}_\varepsilon)(f_R^k(\hat{y}_k,0)-f_L^k(\hat{y}_k,0))\int\limits_{R}\delta_\varepsilon(y_1-y_k) \, dy_k d\hat{y}_k \,dt\\
 &+ \int\limits_{0}^{T}\int\limits_{R^{d-1}}\sgn(u_\varepsilon)(f_R^1(\hat{x}_1,0)-f_L^1(\hat{x}_1,0))\int\limits_R\delta_\varepsilon(x_1-\varphi(\hat{x}_1))\, dx_1\, d\hat{x}_1\,dt 
=0,
    \end{split}
\end{equation*} where $\tilde{u}_\varepsilon(y)={u}_\varepsilon(x)$ and the connection between $x$ and $y$ is given by \eqref{cov}.
This leads us to
\begin{align}
\int\limits_{R^d} |u_\varepsilon|\,dx
\leq T \sum\limits_{k=1}^d\int\limits_{R^{d-1}} |f_L^k(\hat{x}_k,0)-f_R^k(\hat{x}_k,0)|d\hat{x}_k\leq \tilde{C}_T,
\end{align} where the last inequality follows from the integrability assumptions on $f_{L,R}^k(\hat{x}_k,0)$. This completes the proof. $\Box$

In similar way we prove the following lemma:

\begin{lemma} 
\label{Lemma 2.2} 
Let \( u_\varepsilon \) be a solution to \eqref{viscosity}, \eqref{icviscosity}, and assume that 
\eqref{(uslov)} holds. Then
\[
u_\varepsilon(t, x) \leq 0 \quad \text{for a.e. } (t, x) \in \mathbb{R}_+ \times \mathbb{R}^d.
\]

\end{lemma}
{\bf Proof:} The strategy of the proof is similar to that of Lemma \ref{lema1}, with a crucial change: instead of using the entropy function \( \operatorname{sign}(u) \), we now consider the semi-entropy \( \operatorname{sign}_+(u - \lambda) := H(u - \lambda) \), where $H$ is the Heaviside function, and make use of condition \eqref{(uslov)}. This modification results in an equation analogous to \eqref{1.7}.

\begin{align}
\label{(10)}
&\partial_t |u_\varepsilon-\lambda|_+ 
\\&+ \notag \sum\limits_{k=2}^{d}\partial_{x_{k}}\Big(\sgn_+(u_\varepsilon-\lambda)\big((f^k_L(\hat{x}_k,u_\varepsilon)-f_L^k(\hat{x}_k,0)) H_\varepsilon(-x_1+\varphi(\hat{x}_1)) \nonumber \\
&\qquad + (f^k_R(\hat{x}_k,u_\varepsilon)-f_R^k(\hat{x}_k,0)) H_\varepsilon(x_1-\varphi(\hat{x}_1))\big)\Big) \nonumber \\
&+ \sum\limits_{k=2}^d \sgn_+(u_\varepsilon-\lambda)(f_L^k(\hat{x}_k,0)-f_R^k(\hat{x}_k,0))\delta_\varepsilon(x_1-\varphi(\hat{x}_1))\frac{\partial \varphi(\hat{x}_1)}{\partial x_k} \nonumber \\
&+ \partial_{x_{1}}\Big(\sgn_+(u_\varepsilon-\lambda)\big((f^1_L(\hat{x}_1,u_\varepsilon)-f_L^1(\hat{x}_1,0)) H_\varepsilon(-x_1+\varphi(\hat{x}_1)) \nonumber \\
&\qquad + (f^1_R(\hat{x}_1,u_\varepsilon)-f_R^1(\hat{x}_1,0)) H_\varepsilon(x_1-\varphi(\hat{x}_1))\big)\Big) \nonumber \\
&+ \sgn(u_\varepsilon-\varepsilon)_+(f_R^1(\hat{x}_1,0)-f_L^1(\hat{x}_1,0))\delta_\varepsilon(x_1-\varphi(\hat{x}_1)) 
\leq \varepsilon\Delta (u_\varepsilon-\lambda)_+.
\nonumber
\end{align}
By choosing \( \lambda = 0 \), integrating equation \eqref{(10)} over \( \mathbb{R}^d \), and using Lemma \ref{lema1}, which ensures that \( u_\varepsilon(t, \cdot) \in L^1(\mathbb{R}^d) \), we arrive at the following result (see the proof of the previous Lemma):

\begin{align*}
\frac{d}{dt}\int\limits_{R^d} |u_\varepsilon(t,x)|_+\,dx
\leq  \sum\limits_{k=1}^d\int\limits_{R^{d-1}}\sgn_+(u_\varepsilon) (f_L^k(\hat{x}_k,0)-f_R^k(\hat{x}_k,0))d\hat{x}_k.
\end{align*}
From here, since $\sgn_+(u_\varepsilon)\geq 0$ and $f_L^k(\hat{x}_k,0)-f_R^k(\hat{x}_k,0)\leq 0$ by the assumption \eqref{(uslov)}, we get:
\[
\frac{d}{dt} \int\limits_{\mathbb{R}^d} |u_\varepsilon(t, x)|_+ \, dx \leq 0.
\]
By further integrating this result over the interval \( (0, t) \) and noting that \( u_\varepsilon(0, \cdot) \equiv 0 \), we obtain:

\[
\int\limits_{\mathbb{R}^d} |u_\varepsilon(t, x)|_+ \, dx \leq 0,
\] which leads to \[
u_\varepsilon\leq 0 \text{ a.e. } (t,x) \in  \mathbb{R}^+\times\mathbb{R}^d
\]
This completes the argument and establishes the lemma. $\Box$

Before we prove the main theorem of the paper (Theorem \ref{main-thm}) we need yet another lemma.

\begin{lemma} \label{Lemma 2.3}
 For $\varepsilon > 0$, let $u_\varepsilon$ be a solution to \eqref{viscosity}, \eqref{icviscosity}. Furthermore, suppose that an increasing continuous function $a : [0,\infty) \to [0,\infty)$ such that $a(0) = 0$ satisfies $\frac{\varepsilon}{a^2(\varepsilon)} \to 0$ as $\varepsilon\to 0$.
Then, for $\sigma(\varepsilon)=\sqrt[4]{\varepsilon}$ we have:
\begin{equation}
\label{conc}    
\begin{split}
&\limsup_{\varepsilon \to 0} \int\limits_0^t \int\limits_{\mathbb{R}^d} u_\varepsilon(\tau, x) e^{-\big|\frac{x_1-\varphi(\hat{x}_1)}{a(\varepsilon)}\big|_{\sigma(\varepsilon)}} \, dx \, d\tau   \lesssim -{t^2}.
\end{split}
\end{equation}
\end{lemma}
\begin{proof}
In the proof of this proposition we will use the following identities (below $x_k\neq x_1$): 
\begin{equation*}
\begin{split}
&\Big(e^{-\big|\frac{x_1-\varphi(\hat{x}_1)}{a(\varepsilon)}\big|_{\sigma(\varepsilon)}}\Big)'_{x_1}=-\sgn_{\sigma(\varepsilon)}(x_1-\varphi(\hat{x}_1))\frac{1}{a(\varepsilon)}e^{-\big|\frac{x_1-\varphi(\hat{x}_1)}{a(\varepsilon)}\big|_{\sigma(\varepsilon)}},
\\
&\Big(e^{-\big|\frac{x_1-\varphi(\hat{x}_1)}{a(\varepsilon)}\big|_{\sigma(\varepsilon)}}\Big)''_{x_1}=-\delta_{\sigma(\varepsilon)}(x_1-\varphi(\hat{x}_1))\frac{2}{a(\varepsilon)}e^{-\big|\frac{x_1-\varphi(\hat{x}_1)}{a(\varepsilon)}\big|_{\sigma(\varepsilon)}}\\&\qquad\qquad\qquad\qquad \quad \, +\frac{\sgn^2_{\sigma(\varepsilon)}(x_1-\varphi(\hat{x}_1))}{a^2(\varepsilon)}e^{-\big|\frac{x_1-\varphi(\hat{x}_1)}{a(\varepsilon)}\big|_{\sigma(\varepsilon)}},
\\&
\Big(e^{-\big|\frac{x_1-\varphi(\hat{x}_1)}{a(\varepsilon)}\big|_{\sigma(\varepsilon)}}\Big)'_{x_k}=\sgn_{\sigma(\varepsilon)}(x_1-\varphi(\hat{x}_1))\frac{\partial \varphi(\hat{x}_1)}{\partial x_k}\frac{1}{a(\varepsilon)}e^{-\big|\frac{x_1-\varphi(\hat{x}_1)}{a(\varepsilon)}\big|_{\sigma(\varepsilon)}}, 
\\
&\Big(e^{-\left|\frac{x_1-\varphi(\hat{x}_1)}{a(\varepsilon)}\right|_{\sigma(\varepsilon)}}\Big)''_{x_k} 
=\, \left(\frac{\partial \varphi(\hat{x}_1)}{\partial x_k}\right)^2 \frac{1}{a^2(\varepsilon)} 
e^{-\left|\frac{x_1-\varphi(\hat{x}_1)}{a(\varepsilon)}\right|_{\sigma(\varepsilon)}} \\
& \qquad\qquad\qquad\qquad \ \ \,- 2 \delta_{\sigma(\varepsilon)}(x_1-\varphi(\hat{x}_1)) 
\left(\frac{\partial \varphi(\hat{x}_1)}{\partial x_k}\right)^2 \frac{1}{a(\varepsilon)} 
e^{-\left|\frac{x_1-\varphi(\hat{x}_1)}{a(\varepsilon)}\right|_{\sigma(\varepsilon)}} \\
& \qquad\qquad\qquad\qquad \ \ \, + \sgn_{\sigma(\varepsilon)}(x_1-\varphi(\hat{x}_1)) 
\frac{\partial^2 \varphi(\hat{x}_1)}{\partial x_k^2} \frac{1}{a(\varepsilon)} 
e^{-\left|\frac{x_1-\varphi(\hat{x}_1)}{a(\varepsilon)}\right|_{\sigma(\varepsilon)}}.
\end{split}
\end{equation*}  Now we multiply \eqref{viscosity} by $e^{-\big|\frac{x_1-\varphi(\hat{x}_1)}{a(\varepsilon)}\big|_{\sigma(\varepsilon)}}$ and get:
\begin{equation}
\label{MNP}
\begin{split}
&\underbrace{e^{-\big|\frac{x_1 - \varphi(\hat{x}_1)}{a(\varepsilon)}\big|_{\sigma(\varepsilon)}} \, \partial_t u_\varepsilon}_{\text{(M)}} \\
&+ \underbrace{\sum\limits_{k=1}^{d} e^{-|\frac{x_1 \!-\! \varphi(\hat{x}_1)}{a(\varepsilon)}|_{\sigma(\varepsilon)}} 
\partial_{x_k} \left(
f^k_L(\hat{x}_k, u_\varepsilon) H_\varepsilon(-x_1 \!+\! \varphi(\hat{x}_1))
+ f^k_R(\hat{x}_k, u_\varepsilon) H(x_1 \!-\! \varphi(\hat{x}_1))
\right)}_{\text{(N)}} \\
&= \underbrace{\varepsilon e^{-\big|\frac{x_1 - \varphi(\hat{x}_1)}{a(\varepsilon)}\big|_{\sigma(\varepsilon)}} \, \Delta u_\varepsilon}_{\text{(P)}}.
\end{split}
\end{equation}  Then we integrate equation \eqref{MNP} over $[0,T]\times\mathbb{R}^d$. To make the proof easier to follow, we will integrate each term separately. Integrals of the terms $M$, $N$, and $P$ will be denoted respectively by $\bar{M}$, $\bar{N}$, and $\bar{P}$.

{\bf Term M from \eqref{MNP}}

Since $u_\varepsilon(0,x)=0$, we have
\begin{align*}
\bar{M}&=  \int\limits_0^T\int\limits_{\mathbb{R}^d}e^{-\big|\frac{x_1-\varphi(\hat{x}_1)}{a(\varepsilon)}\big|_{\sigma(\varepsilon)}} \partial_t u_\varepsilon\, dt =\int\limits_0^T\int\limits_{\mathbb{R}^d}\partial_t \left(e^{-\big|\frac{x_1-\varphi(\hat{x}_1)}{a(\varepsilon)}\big|_{\sigma(\varepsilon)}}  u_\varepsilon\right)\, dt 
\\&=  \int\limits_{\mathbb{R}^d}e^{-\big|\frac{x_1-\varphi(\hat{x}_1)}{a(\varepsilon)}\big|_{\sigma(\varepsilon)}} u_\varepsilon(T,x)\,dx.
\end{align*}

{\bf Term N from \eqref{MNP}}

\begin{align*}
\bar{N}=&\sum\limits_{k=1}^{d} \int\limits_{0}^{T} \int\limits_{\mathbb{R}^d} 
e^{-\big|\tfrac{x_1 - \varphi(\hat{x}_1)}{a(\varepsilon)}\big|_{\sigma(\varepsilon)}}
\partial_{x_k} \left( 
\begin{aligned}
&f^k_L(\hat{x}_k, u_\varepsilon) H_\varepsilon(-x_1 + \varphi(\hat{x}_1)) \\
&+ f^k_R(\hat{x}_k, u_\varepsilon) H_\varepsilon(x_1 - \varphi(\hat{x}_1))
\end{aligned}
\right) \,dx \, dt \\
=\, &\int\limits_{0}^{T} \int\limits_{\mathbb{R}^d} 
e^{-\big|\tfrac{x_1 - \varphi(\hat{x}_1)}{a(\varepsilon)}\big|_{\sigma(\varepsilon)}}
\partial_{x_1} \left( 
\begin{aligned}
&f^1_L(\hat{x}_1, u_\varepsilon) H_\varepsilon(-x_1 + \varphi(\hat{x}_1)) \\
&+ f^1_R(\hat{x}_1, u_\varepsilon) H_\varepsilon(x_1 - \varphi(\hat{x}_1))
\end{aligned}
\right) \,dx \, dt \\
& +\sum\limits_{k=2}^{d} \int\limits_{0}^{T} \int\limits_{\mathbb{R}^d} 
e^{-\big|\tfrac{x_1 - \varphi(\hat{x}_1)}{a(\varepsilon)}\big|_{\sigma(\varepsilon)}}
\partial_{x_k} \left( 
\begin{aligned}
&f^k_L(\hat{x}_k, u_\varepsilon) H_\varepsilon(-x_1 + \varphi(\hat{x}_1)) \\
&+ f^k_R(\hat{x}_k, u_\varepsilon) H_\varepsilon(x_1 - \varphi(\hat{x}_1))
\end{aligned}
\right) \,dx \, dt
\end{align*}

\begin{align*}
=\, &
-\int\limits_{0}^{T} \int\limits_{\mathbb{R}^d} 
\sgn_{\sigma(\varepsilon)}(x_1 - \varphi(\hat{x}_1)) 
\frac{1}{a(\varepsilon)} e^{-\big|\tfrac{x_1 - \varphi(\hat{x}_1)}{a(\varepsilon)}\big|_{\sigma(\varepsilon)}}
\times \tag{$N_1$} 
\\&\qquad\qquad \times \left( f^1_L(\hat{x}_1, u_\varepsilon) H_\varepsilon(-x_1 + \varphi(\hat{x}_1)) + f^1_R(\hat{x}_1, u_\varepsilon) H_\varepsilon(x_1 - \varphi(\hat{x}_1)) \right)
 \,dx \, dt \\
&+\sum\limits_{k=2}^{d} \int\limits_{0}^{T} \int\limits_{\mathbb{R}^d} 
\sgn_{\sigma(\varepsilon)}(x_1 - \varphi(\hat{x}_1)) 
\frac{\partial \varphi(\hat{x}_1)}{\partial x_k} 
\frac{1}{a(\varepsilon)} e^{-\big|\tfrac{x_1 - \varphi(\hat{x}_1)}{a(\varepsilon)}\big|_{\sigma(\varepsilon)}}\times \tag{$N_2$} \\&\qquad\qquad\times 
\left(f^k_L(\hat{x}_k, u_\varepsilon) H_\varepsilon(-x_1 + \varphi(\hat{x}_1)) + f^k_R(\hat{x}_k, u_\varepsilon) H_\varepsilon(x_1 - \varphi(\hat{x}_1))
\right) \,dx \, dt
\end{align*}
 where in the last equality we have used the integration by parts.

Let us deal with term $N_1$. By introducing the change of variables 
\[\frac{x_1-\varphi(\hat{x}_1)}{a(\varepsilon)}=z,\] with respect to $x_1$,  we get
\begin{equation*}
\begin{aligned}
N_1 = \int\limits_{0}^{T} \int\limits_{\mathbb{R}^{d-1}} \int\limits_{\mathbb{R}} 
& -\sgn_{\sigma(\varepsilon)}(z) \, e^{-|z|_{\sigma(\varepsilon)}} \cdot \Big( 
f_L^1\big(\hat{x}_1, u_\varepsilon(t,a(\varepsilon)z + \varphi(\hat{x}_1),\hat{x}_1)\big) \, H_\varepsilon(-a(\varepsilon)z) \\
&\quad + f_R^1\big(\hat{x}_1, u_\varepsilon(t,a(\varepsilon)z + \varphi(\hat{x}_1),\hat{x}_1)\big) \, H_\varepsilon(a(\varepsilon)z) 
\Big) \, dz \, d\hat{x}_1 \, dt.
\end{aligned}
\end{equation*}
Now we introduce another change of variables, $z \to -z
$ in the summand above containing $H(-a(\varepsilon)z)$ and obtain:
\begin{equation*}
\begin{aligned}
N_1 = \int\limits_{0}^{T} \int\limits_{\mathbb{R}^{d-1}} \int\limits_{0}^{+\infty} 
& \underbrace{\sgn_{\sigma(\varepsilon)}(z)}_{\leq 1} \, e^{-|z|_{\sigma(\varepsilon)}} H_\varepsilon(a(\varepsilon)z) \Big( 
f_R^1\big(\hat{x}_1, u_\varepsilon(t,a(\varepsilon)z + \varphi(\hat{x}_1),\hat{x}_1)\big) \,  \\
&\quad\quad\quad\quad - f_L^1\big(\hat{x}_1, u_\varepsilon(t,-a(\varepsilon)z + \varphi(\hat{x}_1),\hat{x}_1)\big) \,  
\Big) \, dz \, d\hat{x}_1 \, dt.
\end{aligned}
\end{equation*}
Note that by \eqref{(uslov)} 
$$\liminf\limits_{\varepsilon\to 0}N_1\geq \int\limits_{0}^T\int\limits_{\mathbb{R}^{d-1}}\int\limits_{0}^{+\infty}e^{-|z|}\inf\limits_{z_1,z_2\in\mathbb{R}}(f^1_R(\hat{x}_1,z_1)-f_L^1(\hat{x}_1,z_2))\,dz\,d\hat{x}_1\,dt>\tilde{C}T>0. $$
An analogous result holds for the term $N_2$. Indeed, by introducing the change of variables 
\[\frac{x_1-\varphi(\hat{x}_1)}{a(\varepsilon)}=z,\] with respect to $x_k$, $k=2,3,\dots, d$  we get 
\begin{equation*}
\begin{aligned}
N_2 = \sum\limits_{k=2}^{d}\int\limits_{0}^{T} \int\limits_{\mathbb{R}^{d-1}} \int\limits_{\mathbb{R}} 
& -\sgn_{\sigma(\varepsilon)}(z) \, e^{-|z|_{\sigma(\varepsilon)}} \Big( 
f_L^k\big(\hat{x}_k, u_\varepsilon(t,[a(\varepsilon)z + \varphi(\hat{x}_1),\hat{x}_k])\big) \, H_\varepsilon(-a(\varepsilon)z) \\
&\quad + f_R^k\big(\hat{x}_k, u_\varepsilon(t,[a(\varepsilon)z + \varphi(\hat{x}_1),\hat{x}_k])\big) \, H_\varepsilon(a(\varepsilon)z) 
\Big) \, dz \, d\hat{x}_k \, dt.
\end{aligned}
\end{equation*}
Now we again introduce another change of variables, $z \to -z
$ in the summand above containing $H(-a(\varepsilon)z)$ and obtain:
\begin{equation*}
\begin{aligned}
N_2 =\sum\limits_{k=2}^{d} \int\limits_{0}^{T} \int\limits_{\mathbb{R}^{d-1}} \int\limits_{0}^{+\infty} 
& \underbrace{\sgn_{\sigma(\varepsilon)}(z)}_{\leq 1} \, e^{-|z|_{\sigma(\varepsilon)}} H_\varepsilon(a(\varepsilon)z) \Big( 
f_R^k\big(\hat{x}_k, u_\varepsilon(t,[a(\varepsilon)z + \varphi(\hat{x}_1),\hat{x}_k])\big) \,  \\
&\quad\quad\quad\quad - f_L^k\big(\hat{x}_k, u_\varepsilon(t,[-a(\varepsilon)z + \varphi(\hat{x}_1),\hat{x}_k])\big) \,  
\Big) \, dz \, d\hat{x}_k \, dt.
\end{aligned}
\end{equation*}

Again, using \eqref{(uslov)} we conclude that 
\begin{equation*}
\begin{split}
\liminf\limits_{\varepsilon\to 0}N_2&\geq \sum\limits_{k=1}^{d}\int\limits_{0}^T\int\limits_{\mathbb{R}^{d-1}}\int\limits_{0}^{+\infty}e^{-|z|}\inf\limits_{z_1,z_2\in\mathbb{R}}(f^k_R(\hat{x}_k,z_1)-f_L^k(\hat{x}_k,z_2))\, dz\,d\hat{x}_1\,dt \\& \geq\tilde{C}T>0 .
\end{split}
\end{equation*}

Finally, since \( \bar{N} = N_1 + N_2 \), and it was previously shown that
\[
\liminf_{\varepsilon \to 0} N_1 \geq \tilde{C}T > 0
\quad \text{and} \quad
\liminf_{\varepsilon \to 0} N_2 \geq \tilde{C}T > 0,
\]
we conclude that there exists a (generic) constant \( \tilde{C} > 0 \), independent of \( \varepsilon \), such that
\[
\liminf_{\varepsilon \to 0} \bar{N}\geq \tilde{C}T > 0.
\]

{\bf Term P from \eqref{MNP}}

First, note that: 
\begin{equation*}
\begin{split}
&\varepsilon e^{-\big|\frac{x_1-\varphi(\hat{x}_1)}{a(\varepsilon)}\big|_{\sigma(\varepsilon)}} \Delta u_\varepsilon=\varepsilon\sum\limits_{k=1}^{d}e^{-\big|\frac{x_1-\varphi(\hat{x}_1)}{a(\varepsilon)}\big|_{\sigma(\varepsilon)}}\frac{\partial^2u_\varepsilon}{\partial x_k ^2}\\
&=\varepsilon e^{-\big|\frac{x_1-\varphi(\hat{x}_1)}{a(\varepsilon)}\big|_{\sigma(\varepsilon)}}\frac{\partial^2u_\varepsilon}{\partial x_1 ^2}+\varepsilon\sum\limits_{k=2}^{d}e^{-\big|\frac{x_1-\varphi(\hat{x}_1)}{a(\varepsilon)}\big|_{\sigma(\varepsilon)}}\frac{\partial^2u_\varepsilon}{\partial x_k ^2}.
\end{split}
\end{equation*}

Now we have:
\begin{equation}
\begin{split}
&\int\limits_0^T \int\limits_{\mathbb{R}^d} \varepsilon\, 
e^{-\left|\frac{x_1 - \varphi(\hat{x}_1)}{a(\varepsilon)}\right|_{\sigma(\varepsilon)}} 
\frac{\partial^2 u_\varepsilon}{\partial x_1^2} \, dx\,dt 
= -\int\limits_0^T \int\limits_{\mathbb{R}^d} \varepsilon\, 
\left( e^{-\left|\frac{x_1 - \varphi(\hat{x}_1)}{a(\varepsilon)}\right|_{\sigma(\varepsilon)}} \right)'_{x_1} 
\frac{\partial u_\varepsilon}{\partial x_1} \, dx\,dt \\
&= \int\limits_0^T \int\limits_{\mathbb{R}^d} \varepsilon\, 
\left( e^{-\left|\frac{x_1 - \varphi(\hat{x}_1)}{a(\varepsilon)}\right|_{\sigma(\varepsilon)}} \right)''_{x_1} 
u_\varepsilon \, dx\,dt \\
&= -\frac{2\varepsilon}{a(\varepsilon)} \int\limits_0^T \int\limits_{\mathbb{R}^d} 
\delta_{\sigma(\varepsilon)}(x_1 - \varphi(\hat{x}_1))\, 
e^{-\left|\frac{x_1 - \varphi(\hat{x}_1)}{a(\varepsilon)}\right|_{\sigma(\varepsilon)}} 
u_\varepsilon \, dx\,dt \\
&\quad + \frac{\varepsilon}{a^2(\varepsilon)} \int\limits_0^T \int\limits_{\mathbb{R}^d} 
e^{-\left|\frac{x_1 - \varphi(\hat{x}_1)}{a(\varepsilon)}\right|_{\sigma(\varepsilon)}} 
u_\varepsilon \, dx\,dt.
\end{split}
\end{equation} 
Similarly: 
\begin{equation}
\label{3.16}
\begin{split}
&\sum_{k=2}^{d} \int\limits_0^T \int\limits_{\mathbb{R}^d} 
\varepsilon\, e^{-\left|\frac{x_1 - \varphi(\hat{x}_1)}{a(\varepsilon)}\right|_{\sigma(\varepsilon)}} 
\frac{\partial^2 u_\varepsilon}{\partial x_k^2} \, dx \, dt 
\\&= \sum_{k=2}^{d} \Bigg( 
\frac{\varepsilon}{a^2(\varepsilon)} \int\limits_0^T \int\limits_{\mathbb{R}^d} 
\left( \frac{\partial \varphi(\hat{x}_1)}{\partial x_k} \right)^2 
e^{-\left|\frac{x_1 - \varphi(\hat{x}_1)}{a(\varepsilon)}\right|_{\sigma(\varepsilon)}} u_\varepsilon \, dx \, dt \\
&\qquad \ \ - \frac{2\varepsilon}{a(\varepsilon)} \int\limits_0^T \int\limits_{\mathbb{R}^d} 
\delta_{\sigma(\varepsilon)}(x_1 - \varphi(\hat{x}_1)) \left( \frac{\partial \varphi(\hat{x}_1)}{\partial x_k} \right)^2 
e^{-\left|\frac{x_1 - \varphi(\hat{x}_1)}{a(\varepsilon)}\right|_{\sigma(\varepsilon)}} u_\varepsilon \, dx \, dt \\
&\qquad\ \ + \frac{\varepsilon}{a(\varepsilon)} \int\limits_0^T \int\limits_{\mathbb{R}^d} 
\sgn_{\sigma(\varepsilon)}(x_1 - \varphi(\hat{x}_1)) \frac{\partial^2 \varphi(\hat{x}_1)}{\partial x_k^2} 
e^{-\left|\frac{x_1 - \varphi(\hat{x}_1)}{a(\varepsilon)}\right|_{\sigma(\varepsilon)}} u_\varepsilon \, dx \, dt 
\Bigg).
\end{split}
\end{equation} Combining the results from the above we get:

\begin{equation*}
\begin{split}
\bar{P}=&\int\limits_{0}^{T}\int\limits_{\mathbb{R}^d}\varepsilon e^{-\big|\frac{x_1 - \varphi(\hat{x}_1)}{a(\varepsilon)}\big|_{\sigma(\varepsilon)}} \, \Delta u_\varepsilon \,dx\,dt\\
&= \underbrace{-\frac{2\varepsilon}{a(\varepsilon)} 
\int\limits_{0}^{T} \int\limits_{\mathbb{R}^d} 
\delta_{\sigma(\varepsilon)}(x_1 - \varphi(\hat{x}_1))\, 
e^{-\big|\frac{x_1 - \varphi(\hat{x}_1)}{a(\varepsilon)}\big|_{\sigma(\varepsilon)}} 
u_\varepsilon \, dx\,dt}_{P_1} \\
&\quad + \underbrace{\frac{\varepsilon}{a^2(\varepsilon)} 
\int\limits_{0}^{T} \int\limits_{\mathbb{R}^d} 
e^{-\big|\frac{x_1 - \varphi(\hat{x}_1)}{a(\varepsilon)}\big|_{\sigma(\varepsilon)}} 
u_\varepsilon \, dx\,dt}_{P_2}
\end{split}
\end{equation*}

\begin{equation}\label{1.8}
\begin{split}
&\quad+ \sum\limits_{k=2}^{d} \Bigg( \underbrace{  
\frac{\varepsilon}{a^2(\varepsilon)} 
\int\limits_{0}^{T} \int\limits_{\mathbb{R}^d} 
\left( \frac{\partial \varphi(\hat{x}_1)}{\partial x_k} \right)^2 
e^{-\big|\frac{x_1 - \varphi(\hat{x}_1)}{a(\varepsilon)}\big|_{\sigma(\varepsilon)}} 
u_\varepsilon \, dx \, dt}_{P_3} \\
&\qquad - \underbrace{\frac{2\varepsilon}{a(\varepsilon)} 
\int\limits_{0}^{T} \int\limits_{\mathbb{R}^d} 
\delta_{\sigma(\varepsilon)}(x_1 - \varphi(\hat{x}_1)) 
\left( \frac{\partial \varphi(\hat{x}_1)}{\partial x_k} \right)^2 
e^{-\big|\frac{x_1 - \varphi(\hat{x}_1)}{a(\varepsilon)}\big|_{\sigma(\varepsilon)}} 
u_\varepsilon \, dx \, dt}_{P_4} \\
&\qquad + \underbrace{\frac{\varepsilon}{a(\varepsilon)} 
\int\limits_{0}^{T} \int\limits_{\mathbb{R}^d} 
\sgn_{\sigma(\varepsilon)}(x_1 - \varphi(\hat{x}_1)) 
\frac{\partial^2 \varphi(\hat{x}_1)}{\partial x_k^2} 
e^{-\big|\frac{x_1 - \varphi(\hat{x}_1)}{a(\varepsilon)}\big|_{\sigma(\varepsilon)}} 
u_\varepsilon\, dx \, dt}_{P_5} 
\Bigg).
\end{split}
\end{equation}

{Note that} 
\begin{itemize}
    \item $|P_1|\lesssim \frac{\varepsilon}{a(\varepsilon)\sigma(\varepsilon)},$
    \item $|P_2|\lesssim \frac{\varepsilon}{a^2(\varepsilon)},$
    \item $|P_3|\lesssim \frac{\varepsilon}{a^2(\varepsilon)},$
    \item $|P_4|\lesssim \frac{\varepsilon}{a(\varepsilon)\sigma(\varepsilon)},$
    \item $|P_5|\lesssim \frac{\varepsilon}{a(\varepsilon)}.$
\end{itemize}
Let us prove, for instance, that
\[
|P_1| \lesssim \frac{\varepsilon}{a(\varepsilon)\sigma(\varepsilon)}.
\]
We note that
\[ e^{ - \big| \frac{x_1 - \varphi(\hat{x}_1)}{a(\varepsilon)} \big|_{\sigma(\varepsilon)} }  \leq 1.
\]
Furthermore, Lemma~\ref{lema1} provides the bound
\[
\| u_\varepsilon(t,\cdot) \|_{L^1(\mathbb{R}^d)} \leq \tilde{C}_T.
\]
In addition, since $\delta_{\sigma(\varepsilon)}$ is an approximation of the Dirac delta distribution, it must satisfy
\[
\left| \delta_{\sigma(\varepsilon)} \right| \lesssim \frac{1}{\sigma(\varepsilon)}.
\]

Combining all the above estimates, we conclude that
\[
|P_1| \lesssim \frac{\varepsilon}{a(\varepsilon)\sigma(\varepsilon)}.
\]
The remaining four relations are proved in a similar way.

Taking this into account, 
we obtain that the entire Term $\bar{P}$ tends to zero as $\varepsilon \to 0$:
$$
\lim\limits_{\varepsilon \to 0} \bar{P}=0.
$$

\medskip

Finally, it remained to collect the terms $\bar{M}$, $\bar{N}$ and $\bar{P}$. 
We have the following relation:
\begin{equation}\label{zaintegraljenje}
\bar{M} + \bar{N} = \bar{P}.
\end{equation}

By integrating equation \eqref{zaintegraljenje} over the interval $[0, t]$ with respect to the variable $T$, and taking into account that $\liminf\limits_{\varepsilon \to 0} \bar{N} \geq \tilde{C}T$ and $\lim\limits_{\varepsilon \to 0} \bar{P} = 0$, and by using the Fatou lemma, we obtain:

\begin{align*}
\limsup\limits_{\varepsilon \to 0} \int\limits_{0}^t \int\limits_{\mathbb{R}^d} 
e^{-\left|\frac{x_1 - \varphi(\hat{x}_1)}{a(\varepsilon)}\right|_{\sigma(\varepsilon)}} 
u_\varepsilon(T, x) \, dx \, dT 
&\leq -\liminf\limits_{\varepsilon \to 0} \int\limits_{0}^t \bar{N}(T) \, dT \notag \\
&\leq -\int\limits_0^t\liminf\limits_{\varepsilon\to 0}\bar{N}(T)\, dT 
\lesssim -\int\limits_0^t T\, dT,
\end{align*}

which leads us to the final conclusion:

\begin{equation}
\label{conc-}
\limsup\limits_{\varepsilon\to 0} \int\limits_{0}^t\int\limits_{\mathbb{R}^d} 
e^{-\big|\frac{x_1 - \varphi(\hat{x}_1)}{a(\varepsilon)}\big|_{\sigma(\varepsilon)}} 
u_\varepsilon(T,x) \, dx\, dT \lesssim-t^2.
\end{equation} The proof of the lemma is completed. $\Box$

Now, we can prove the main theorem of the paper.

\newpage

{\bf Proof of Theorem \ref{main-thm}}

\medskip

Since, by Lemma \ref{lema1},  $(u_\varepsilon)$ is bounded in $L^1([0,T]\times \mathbb{R}^d)$ it follows that along a subsequence, we have $$u_\varepsilon \rightharpoonup u \text{ in } \mathcal{M}((0,T)\times \mathbb{R}^d).$$
Drawing from Lemma \ref{Lemma 2.2}, we see that the Radon measure $u$ is such that $u \leq 0$.
To proceed, denote for $\varepsilon>0$ 
$$
A_\varepsilon=\{x\in R^d: \, |x_1-\varphi(\hat{x}_1)|<\varepsilon \}.
$$

Now we want to show that for any \( t \in (0, T] \) and \( \varepsilon > 0 \), the measure \( u \) does not belong to \( L^1((0,t) \times A_\varepsilon) \). Specifically, since \( u \) is a Radon measure (and thus outer regular) and the sets \( U_\varepsilon := (0,t) \times A_\varepsilon \) are open, we can express
\begin{equation}
\label{D}
u((0,t) \times \{x_1=\varphi(\hat{x}_1)\}) = \lim_{\varepsilon \to 0} u(A_\varepsilon).
\end{equation}

Moreover, for each \( \varepsilon > 0 \), by the weak lower semicontinuity (see Theorem \ref{Ajlan}),
\begin{equation} \label{eq:2.15}
u(A_\varepsilon) \leq \liminf_{\varepsilon \to 0} u_\varepsilon(A_\varepsilon) = \liminf\limits_{\varepsilon \to 0} \int\limits_0^T \int\limits_{A_\varepsilon} u_\varepsilon(t, x) \, dx \, dt.
\end{equation}
Since \( u_\varepsilon \) is non-positive and bounded (albeit non-uniformly in \( \varepsilon \)), and the function \( e^{-\big|\frac{x_1-\varphi(\hat{x}_1)}{a(\varepsilon)}\big|_{\sigma(\varepsilon)}} \) lies strictly between 0 and 1, the following inequality holds:
\[
\int\limits_0^T \int\limits_{A_\varepsilon} u_\varepsilon(t, x) \, dx \, dt \leq \int\limits_0^T \int\limits_{\mathbb{R}^d} u_\varepsilon(t, x) e^{-\big|\frac{x_1-\varphi(\hat{x}_1)}{a(\varepsilon)}\big|_{\sigma(\varepsilon)}} \, dx \, dt.
\]
Thus, combining the latter with \eqref{eq:2.15} and \eqref{conc-}, we find:
\[
u(A_\varepsilon) \leq \liminf_{\varepsilon \to 0} \int\limits_0^t \int\limits_{\mathbb{R}} u_\varepsilon(t, x) e^{-\big|\frac{x_1-\varphi(\hat{x}_1)}{a(\varepsilon)}\big|_{\sigma(\varepsilon)}} \, dx \, dt \leq -\frac{\tilde{C}t^2}{2} < 0,
\]
for any \( \delta > 0 \). This leads us to conclude that from \eqref{D}
\[
u((0,t) \times \{x_1=\varphi(\hat{x}_1)\}) < 0,
\]
thereby completing the proof of the theorem.

\end{proof}

\begin{remark}
In the case when the flux satisfies the local non-alignement conditions:
\begin{equation}
   \label{(uslov-0)}
\int\limits_{\mathbb{R}^{d-1}}\big(\min_{|\lambda| \leq M } f^k_L(\hat{x}_k,\lambda) - \max_{|\lambda| \leq M } f^k_R(\hat{x}_k, \lambda) \big)\, d\hat{x}_k<C(M), \ \ \forall k=1,\dots,d,
\end{equation} implying that $f_L^k$ and $f_R^k$ possibly become equal at $\infty$ thus making the constant $c$ in \eqref{(uslov)} equals zero (see \cite[Figure 1, C)]{krunic2021} in one dimensional situation), we are only able to conclude that that $(u_\varepsilon)$ is unbounded, but we do not now whether it converges to a singular measure.

Indeed, if we assume that $\sup_\varepsilon|u_\varepsilon|\leq M$, then it is enough to consider the flux components $f^k_L(\hat{x}_k,\lambda)$ and $f^k_R(\hat{x}_k,\lambda)$, $k=1,\dots,d$, for $\lambda\in [-M,M]$ i.e. we can replace it in \eqref{viscosity} with the flux components
$$
\tilde{f}^k_L(\hat{x}_k,\lambda)=\begin{cases}
    {f}^k_L(\hat{x}_k,\lambda), & |\lambda|<M,\\
    {f}^k_L(\hat{x}_k,M), & |\lambda|\geq M,
\end{cases}
$$ and
$$
\tilde{f}^k_R(\hat{x}_k,\lambda)=\begin{cases}
    {f}^k_R(\hat{x}_k,\lambda), & |\lambda|<M,\\
    {f}^k_R(\hat{x}_k,M), & |\lambda|\geq M.
\end{cases}.
$$ Now, \eqref{(uslov)} for the functions  $\tilde{f}^k_L$ and $\tilde{f}^k_R$ becomes
\begin{equation}
   \label{(uslov-1)}
   \begin{split}
&\int\limits_{\mathbb{R}^{d-1}}\big(\min_{\lambda \in \mathbb{R}} \tilde{f}^k_L(\hat{x}_k,\lambda) - \max_{\lambda \in \mathbb{R}} \tilde{f}^k_R(\hat{x}_k, \lambda) \big)\, d\hat{x}_k
\\&=\int\limits_{\mathbb{R}^{d-1}}\big(\min_{|\lambda| \leq M} \tilde{f}^k_L(\hat{x}_k,\lambda) - \max_{|\lambda| \leq M} \tilde{f}^k_R(\hat{x}_k, \lambda) \big)\, d\hat{x}_k
\\&=\int\limits_{\mathbb{R}^{d-1}}\big(\min_{|\lambda| \leq M} {f}^k_L(\hat{x}_k,\lambda) - \max_{|\lambda| \leq M} {f}^k_R(\hat{x}_k, \lambda) \big)\, d\hat{x}_k \overset{\eqref{(uslov-0)}}{\leq } C(M).
\end{split}
\end{equation} However, under such a condition, the vanishing viscosity approximation \eqref{viscosity}, \eqref{icviscosity} will generate the family of solutions admitting a subsequence tending to a singular measure contradicting the assumption on boundedness of $(u_\varepsilon)$.
\end{remark}

\section{Concluding Remarks}

In this work, we analyzed a multidimensional scalar conservation law with spatially discontinuous and non-aligned flux, subject to zero initial data. By employing the vanishing viscosity method and precise control of interface geometry, we have shown that the limit of the approximating solutions does not converge to a classical function, nor even to an $L^1$-function, but rather to a singular Radon measure supported on the flux discontinuity surface.

This result reveals a highly nontrivial and somewhat unexpected phenomenon: the generation of a singular measure in the limit of a scalar conservation law, despite the absence of any nonlinear source terms or external concentration mechanisms. Typically, scalar conservation laws with discontinuous flux are studied under the assumption that weak entropy solutions belong to $L^1$ or are at least measure-valued in a more regular sense. The emergence of a measure that is strictly singular—concentrated along a codimension-one manifold-challenges this intuition.

Such singular limits have previously been observed in more structured or system-type models, including pressureless gas dynamics and sticky particle systems, where delta concentrations arise due to linearity in one of the unknowns or due to particle collision mechanisms. In contrast, the scalar nature and purely flux-based discontinuity in our setting offer no immediate mechanism for such concentration, making the appearance of a singularity both subtle and surprising.

Moreover, since the flux function is nonlinear in the unknown $u$, the limiting measure cannot be simply substituted into the flux to yield a meaningful weak formulation. This raises deep questions about the proper framework for interpreting the limit: whether in the sense of generalized functions (e.g., Colombeau algebras), duality solutions, or some other extension of the standard theory.

Our result suggests that singularity formation in scalar conservation laws with discontinuous flux is a more prevalent and structurally stable phenomenon than previously thought, especially in multidimensional geometries with non-aligned interfaces. This invites further research into the characterization, uniqueness, and physical interpretation of such singular solutions, as well as the development of numerical schemes capable of capturing them correctly.

\section{Acknowledgement}
The author would like to thank D. Mitrovic for suggesting the problem and for his valuable assistance throughout the writing of this paper.
The work of the author is supported in part by the projects "Mathematical Analysis, Optimization, and Machine Learning", funded by Ministry of Education, Science and Innovation of Montenegro and "Methods of optimization, solving variational and quasi-variational inequalities", funded by Montenegrin Academy of Sciences and Arts.

\end{document}